\documentclass[a4paper,fleqn]{cas-sc}

\usepackage[numbers,longnamesfirst]{natbib}
\usepackage{amsmath,amssymb,amsthm,mathtools,mathrsfs,bm}
\usepackage{hyperref}
\usepackage[capitalize,nameinlink]{cleveref}
\usepackage{enumitem}
\usepackage{graphicx}
\usepackage{booktabs}
\usepackage[section]{placeins}
\usepackage{microtype}
\usepackage{dsfont}
\usepackage{float}

\hypersetup{colorlinks=true,citecolor=blue,linkcolor=blue,urlcolor=blue,hypertexnames=false}
\RenewDocumentCommand \printorcid {} {}

\newtheorem{theorem}{Theorem}[section]
\newtheorem{proposition}[theorem]{Proposition}
\newtheorem{corollary}[theorem]{Corollary}
\newtheorem{lemma}[theorem]{Lemma}
\theoremstyle{definition}
\newtheorem{remark}[theorem]{Remark}

\newtheorem{example}[theorem]{Example}
\newcommand{\R}{\mathbb{R}}
\newcommand{\Sp}{\mathbb{S}_{++}}
\newcommand{\Sym}{\mathbb{S}}
\newcommand{\tr}{\operatorname{tr}}
\newcommand{\diag}{\operatorname{diag}}
\newcommand{\ip}[2]{\left\langle #1,#2\right\rangle}
\newcommand{\norm}[1]{\left\lVert #1\right\rVert}
\newcommand{\MN}{\mathcal{MN}}
\newcommand{\mcM}{\mathcal{M}}
\newcommand{\mcK}{\mathcal{K}}
\newcommand{\mcJ}{\mathcal{J}}
\newcommand{\mcL}{\mathscr{L}}

\newcommand{\ot}{\otimes}
\newcommand{\argmin}{\operatorname*{argmin}}
\newcommand{\tabnote}[1]{\par\smallskip\begin{minipage}{0.96\textwidth}\footnotesize\raggedright\emph{Note.}~#1\end{minipage}}

\def\tsc#1{\csdef{#1}{\textsc{\lowercase{#1}}\xspace}}
\tsc{WGM}
\tsc{QE}

\begin{document}
\let\WriteBookmarks\relax
\def\floatpagepagefraction{.35}
\def\textpagefraction{.08}

\shorttitle{Bures Geodesics for Kronecker Positive Definite Matrices}    

\shortauthors{Yang, J.,  Zhang, Y.}  

\title [mode = title]{Bures Geodesics and Restricted Barycenters for Kronecker Positive Definite Matrices}  



%


%

\author[1]{Jiaping Yang}

\cormark[1]


\ead{jpyang22@m.fudan.edu.cn}



\affiliation[1]{organization={School of Mathematical Sciences, Fudan University},
	city={Shanghai},
	postcode={200433}, 
	country={China}}

\author[1]{Yunxin Zhang}


\ead{xyz@fudan.edu.cn}




\cortext[1]{Corresponding author}
















\begin{abstract}
We study the extrinsic Bures--Wasserstein geometry of the determinant-normalized Kronecker model $\mcK_n=\{V\ot U:U,V\in\Sp^n,\ \det U=1\}\subset\Sp^{n^2}$, asking when the ambient Bures geodesic between two Kronecker positive definite matrices can remain in this lower-dimensional model. Local membership near an endpoint is shown to be equivalent to membership of the whole segment, and this happens exactly in the one-factor cases: either $U_1=U_0$ or $V_1$ is a positive scalar multiple of $V_0$. Consequently, any endpoint pair not confined to these one-factor alternatives leaves the model immediately. The criterion is expressed by a partial-trace residual. In fixed commuting charts it becomes an equivalent rank-one square-root profile and yields computable departure diagnostics. We also obtain exact formulas for two restricted barycenter problems: fixed commuting-coordinate slices, solved by Perron singular vectors, and one-factor subfamilies, reduced to standard Bures--Wasserstein barycenters on $\Sp^n$.
\end{abstract}




\maketitle

\section{Introduction}
Kronecker products of positive definite matrices are basic objects in matrix analysis and numerical linear algebra \cite{HornJohnson2013,VanLoan2000}. They also provide the covariance structure in separable models for array-valued data \cite{Hoff2011,WernerJanssonStoica2008,StegleEtAl2011,DrtonKurikiHoff2021}. We work with the determinant-normalized Kronecker model
\[
\mcK_n=\{V\ot U:U,V\in\Sp^n,\ \det U=1\}\subset\Sp^{n^2}.
\]
The gauge condition removes the scalar ambiguity $(cV)\ot(c^{-1}U)=V\ot U$ and gives unique factor coordinates. The apparent asymmetry in some statements is only a consequence of placing the normalization on the $U$-factor.

The Bures--Wasserstein geometry of the full positive-definite cone is well understood, but an ambient geodesic has no reason to respect a lower-dimensional matrix constraint. The question addressed here is extrinsic: which Kronecker endpoint pairs generate an ambient Bures segment that stays in $\mcK_n$, even locally at an endpoint? The answer is rigid. Endpoint-local retention is equivalent to whole-segment retention, and both occur only in the two one-factor cases: either $U_1=U_0$ or $V_1$ is a positive scalar multiple of $V_0$. Outside these alternatives, the ambient geodesic leaves the Kronecker model immediately.

We call these one-factor subfamilies factor leaves: they are the subfamilies where one normalized factor is fixed, up to the scalar gauge. The analysis develops several matrix-analytic reductions. A pairwise spectral formula expresses Bures distances between Kronecker products through two factor-size computations. For noncommuting endpoints, a partial-trace residual for the whitened initial velocity detects endpoint closure and leads to the factor-leaf rigidity theorem. In fixed commuting charts, the same obstruction becomes a rank-one square-root profile, giving departure moduli and excluding isolated interior returns. For barycenters, the scope is restricted to settings where exact formulas are available: a fixed commuting-coordinate slice has an explicit Perron singular-vector minimizer, while a common leaf reduces to the standard Bures--Wasserstein barycenter problem on $\Sp^n$ \cite{AguehCarlier2011,AlvarezEstebanDelBarrioCuestaAlbertosMatran2016}. Notably, no global barycenter formula on all of $\mcK_n$ is claimed.

The closest matrix-analysis references are the tensor-product Wasserstein mean identities in \cite{HwangKim2020} and the linearity problem for Cartan and Wasserstein means in \cite{ChoiKimLim2024}. Those works concern mean operations and tensor-product structure. Here the Kronecker parametrization is instead treated as an embedded submanifold of the full cone, and the central question is preservation by the ambient Bures geometry. Background on the Bures--Wasserstein distance can be found in \cite{Bures1969,DowsonLandau1982,OlkinPukelsheim1982,Gelbrich1990,BhatiaJainLim2019}; related statistical motivation appears in recent work on Kronecker covariance geometry and testing \cite{GuggenbergerKleibergenMavroeidis2023,YuXieZhou2023,McCormackHoff2025,SimonisWells2025}.

The remainder of this paper is organized as follows.
\Cref{sec:preliminaries} introduces the normalized model and the pairwise reduction. \Cref{sec:geodesicclosure} proves the geodesic closure results and endpoint rigidity theorem. \Cref{sec:barycenters} treats fixed-coordinate-slice and leafwise barycenters. Numerical checks appear in \Cref{sec:numerics}, followed by concluding remarks in \Cref{sec:con}.

\section{Preliminaries}
\label{sec:preliminaries}

\subsection{Determinant-normalized Kronecker model}
Throughout, assume $n\ge2$ unless stated otherwise.  Our main object is the submanifold $\mcK_n\subset\Sp^{n^2}$. The matrix-normal notation below records the Gaussian interpretation of the pairwise formula; the structural results themselves are statements about positive definite matrices.
All matrix inner products and norms are Frobenius unless another norm is explicitly indicated:
\[
\ip{X}{Y}_F:=\tr(X^\top Y),\qquad \norm{X}_F^2:=\ip{X}{X}_F .
\]

A matrix $K\in\R^{n^2\times n^2}$ is viewed as an $n\times n$ block matrix $K=[K_{qr}]_{q,r=1}^n$ with blocks $K_{qr}\in\R^{n\times n}$, consistent with the convention that $V\ot U$ has $(q,r)$ block $v_{qr}U$.  The partial traces used below are
\[
\operatorname{tr}_1(K):=\sum_{q=1}^n K_{qq},
\qquad
\operatorname{tr}_2(K):=\bigl[\tr(K_{qr})\bigr]_{q,r=1}^n .
\]
Hence $\operatorname{tr}_1(V\ot U)=\tr(V)U$ and $\operatorname{tr}_2(V\ot U)=\tr(U)V$.

The matrix-normal notation $X\sim \MN(M,U,V)$ means $\operatorname{vec}(X)\sim \mathcal N(\operatorname{vec}(M),V\ot U)$. Accordingly, matrix-normal laws are identified with their vectorized Gaussian representations.

The factorization is determined only up to reciprocal scaling, since $(cV)\ot(c^{-1}U)=V\ot U$ for $c>0$. Imposing the gauge $\det U=1$ removes this scalar indeterminacy; set
\[
\mcM_n:=\{(U,V)\in \Sp^n\times \Sp^n:\det U=1\},
\qquad
\Phi(U,V):=V\ot U.
\]
The image $\mcK_n:=\Phi(\mcM_n)\subset \Sp^{n^2}$ is the \emph{determinant-normalized Kronecker positive-definite model}.

\begin{lemma}
	\label{lem:modelmanifold}
	The map $\Phi:\mcM_n\to\Sp^{n^2}$ is a smooth embedding. Therefore, $\mcK_n$ is a smooth embedded submanifold of $\Sp^{n^2}$ of dimension $n(n+1)-1$. Moreover,
	\[
	T_{(U,V)}\mcM_n
	=
	\{(H_U,H_V)\in\Sym^n\times\Sym^n:\tr(U^{-1}H_U)=0\},
	\]
	and
	\[
	D\Phi_{(U,V)}[H_U,H_V]=V\ot H_U+H_V\ot U.
	\]
\end{lemma}

\begin{proof}
	The manifold $\mcM_n$ is the regular level set of $(U,V)\mapsto \log\det U$, which induces the tangent space. Differentiating $\Phi(U,V)=V\ot U$ produces the displayed differential.
	
	For injectivity, suppose $V\ot U=\widetilde V\ot\widetilde U$ with both pairs in $\mcM_n$.  Comparing the $(1,1)$ blocks yields
	$v_{11}U=\widetilde v_{11}\widetilde U$, where $v_{11},\widetilde v_{11}>0$. So $\widetilde U=cU$ for some $c>0$, and comparison of all blocks gives $V=c\widetilde V$. The determinant constraints force $c=1$.
	
	The inverse on the image is explicit. If $K=V\ot U$ is written in normalized form, then
	$\operatorname{tr}_1(K)=\tr(V)\,U$ and $\operatorname{tr}_2(K)=\tr(U)\,V$. Since $\det U=1$, we have $U=\det(\operatorname{tr}_1(K))^{-1/n}\operatorname{tr}_1(K)$ and $V=\tr(U)^{-1}\operatorname{tr}_2(K)$. Therefore, $\Phi^{-1}$ is continuous on $\mcK_n$.
	
	It only remains to verify immersion. Let
	$V\ot H_U+H_V\ot U=0$ and $\tr(U^{-1}H_U)=0$. Multiplying by $I\ot U^{-1}$ and taking the second partial trace yields $nH_V=0$, whence $H_V=0$ and hence $H_U=0$.
\end{proof}

\subsection{Pairwise Kronecker reduction}
\label{sec:pairwise}
For Gaussian laws on $\R^d$, the squared $W_2$ distance is given by
\begin{equation}
	\label{eq:GaussianW2}
	W_2^2\bigl(\mathcal N(m_0,K_0),\mathcal N(m_1,K_1)\bigr)
	=
	\norm{m_0-m_1}^2
	+\tr(K_0)+\tr(K_1)
	-2\tr\bigl((K_0^{1/2}K_1K_0^{1/2})^{1/2}\bigr).
\end{equation}
The covariance term in \eqref{eq:GaussianW2} is the squared Bures--Wasserstein distance on $\Sp^d$,
\[
d_{\mathrm B}^2(A,B)
:=
\tr(A)+\tr(B)-2\tr\bigl((A^{1/2}BA^{1/2})^{1/2}\bigr),
\qquad A,B\in\Sp^d.
\]

The Bures geodesic from $A$ to $B$ is
\begin{equation}
	\label{eq:BuresGeoGeneral}
	\gamma_{A\to B}(t)=\bigl((1-t)I+tT_{A\to B}\bigr)A\bigl((1-t)I+tT_{A\to B}\bigr),
	\qquad t\in[0,1],
\end{equation}
where $T_{A\to B}=A^{-1/2}(A^{1/2}BA^{1/2})^{1/2}A^{-1/2}$.
The transport is positive definite and satisfies $T_{A\to B}AT_{A\to B}=B$, so \eqref{eq:BuresGeoGeneral} indeed joins $A$ to $B$.
When $A$ and $B$ commute,
\begin{equation}
	\label{eq:BuresGeoCommuting}
	\gamma_{A\to B}(t)=\bigl((1-t)A^{1/2}+tB^{1/2}\bigr)^2.
\end{equation}
These formulas are standard; see \cite{DowsonLandau1982,OlkinPukelsheim1982,Gelbrich1990,Takatsu2011,BhatiaJainLim2019}.

The following reduction will be used repeatedly. It is a direct consequence of standard Kronecker product spectral identities \cite{HornJohnson2013,VanLoan2000}, and it replaces one ambient Bures computation on $\Sp^{n^2}$ by two factor-size spectral computations on $\Sp^n$.

\begin{lemma}
	\label{lem:pairwiseW2}
	Let $K_i=V_i\ot U_i$ with $U_i,V_i\in\Sp^n$, $i=0,1$, and set
	\[
	A:=V_0^{1/2}V_1V_0^{1/2},
	\qquad
	B:=U_0^{1/2}U_1U_0^{1/2}.
	\]
	Suppose the eigenvalues of $A$ and $B$ are $\alpha_1,\dots,\alpha_n$ and $\beta_1,\dots,\beta_n$, respectively. Then
	\begin{equation}
		\label{eq:pairwiseBuresKronecker}
		d_{\mathrm B}^2(K_0,K_1)
		=
		\tr(U_0)\tr(V_0)+\tr(U_1)\tr(V_1)
		-2\sum_{p=1}^n\sum_{q=1}^n \sqrt{\alpha_p\beta_q}.
	\end{equation}
	Equivalently,
	\[
	d_{\mathrm B}^2(K_0,K_1)
	=
	\tr(U_0)\tr(V_0)+\tr(U_1)\tr(V_1)
	-2\tr(A^{1/2})\tr(B^{1/2}).
	\]
	For Gaussian laws $\mu_i=\mathcal N(m_i,K_i)$, $i=0,1$, this gives $W_2^2(\mu_0,\mu_1)=\norm{m_0-m_1}^2+d_{\mathrm B}^2(K_0,K_1)$. For matrix-normal laws $\mu_i=\MN(M_i,U_i,V_i)$, $i=0,1$, we have
	\begin{equation}
		\label{eq:pairwiseW2Kronecker}
		W_2^2(\mu_0,\mu_1)
		=
		\norm{M_0-M_1}_F^2
		+\tr(U_0)\tr(V_0)+\tr(U_1)\tr(V_1)
		-2\sum_{p=1}^n\sum_{q=1}^n \sqrt{\alpha_p\beta_q}.
	\end{equation}
\end{lemma}

\begin{proof}
	The Bures formula yields
	\[
	d_{\mathrm B}^2(K_0,K_1)
	=
	\tr(K_0)+\tr(K_1)-2\tr\bigl((K_0^{1/2}K_1K_0^{1/2})^{1/2}\bigr),
	\]
	where
	\[
	K_0^{1/2}K_1K_0^{1/2}
	=
	(V_0^{1/2}V_1V_0^{1/2})\ot(U_0^{1/2}U_1U_0^{1/2})=A\ot B.
	\]
	The eigenvalues of this Kronecker product are $\alpha_p\beta_q$, whence $\tr\bigl((K_0^{1/2}K_1K_0^{1/2})^{1/2}\bigr)=\sum_{p,q}\sqrt{\alpha_p\beta_q}$. Together with $\tr(V_i\ot U_i)=\tr(V_i)\tr(U_i)$ we obtain \eqref{eq:pairwiseBuresKronecker}. The Gaussian and matrix-normal formulas follow from \eqref{eq:GaussianW2} after identifying $m_i=\operatorname{vec}(M_i)$.
\end{proof}

\begin{remark}
	\label{rem:pairwisecomplexity}
	For dense matrices, a naive ambient evaluation of the Bures term would require a square-root or spectral computation for an $n^2\times n^2$ positive-definite matrix, which means $O(n^6)$ work and $O(n^4)$ storage. By \Cref{lem:pairwiseW2}, the reduced formula requires only two $n\times n$ positive-definite spectral computations, hence $O(n^3)$ work and $O(n^2)$ storage.
\end{remark}

\section{Geodesic closure and endpoint tangency}
\label{sec:geodesicclosure}

We now characterize endpoint pairs for which the ambient Bures segment is retained by the determinant-normalized Kronecker model. Without commutativity, retention is understood at an endpoint: the segment must stay in $\mcK_n$ on a nontrivial interval starting there. The endpoint-local statements below are written at $K_0$; the corresponding assertions at $K_1$ follow by reversing the two endpoints. The rigidity theorem below shows that this local condition is already equivalent to whole-segment closure. In a fixed commuting chart one can say more, because the same factor-leaf criterion also excludes isolated interior returns and leads to explicit departure moduli.

Choose $U_\ast\in\Sp^n$ with $\det U_\ast=1$ and $V_\ast\in\Sp^n$, and write
\[
\mathcal F(U_\ast):=\{V\ot U_\ast:V\in\Sp^n\},
\qquad
\mathcal G(V_\ast):=\{\alpha V_\ast\ot U:\alpha>0,\ U\in\Sp^n,\ \det U=1\},
\]
for the two factor leaves through a point of the model. Their isotropic representatives are the canonical isotropic factor leaves
\[
\mathcal L_{\mathrm{row}}:=\{\alpha I_n\ot U:\alpha>0,\ U\in\Sp^n,\ \det U=1\},
\qquad
\mathcal L_{\mathrm{col}}:=\{V\ot I_n:V\in\Sp^n\}.
\]
Here $\mathcal F(U_\ast)$ fixes the determinant-normalized $U$-factor, whereas $\mathcal G(V_\ast)$ fixes the $V$-factor up to the scalar absorbed by the determinant gauge.

\subsection{Geodesic closure of factor leaves}

\begin{lemma}
	\label{lem:geodesicallyclosedleaves}
	For every $U_\ast\in\Sp^n$ with $\det U_\ast=1$, the leaf $\mathcal F(U_\ast)=\{V\ot U_\ast:V\in\Sp^n\}$ is geodesically closed in the ambient Bures geometry. More precisely, if $K_i=V_i\ot U_\ast$, $i=0,1$, then the ambient Bures geodesic satisfies
	\[
	\gamma_{K_0\to K_1}(t)=\gamma_{V_0\to V_1}(t)\ot U_\ast,\qquad t\in[0,1].
	\]
	For every $V_\ast\in\Sp^n$, the leaf $\mathcal G(V_\ast)=\{\alpha V_\ast\ot U:\alpha>0,\ U\in\Sp^n,\ \det U=1\}$ is also geodesically closed. If $K_i=\alpha_iV_\ast\ot U_i$, $i=0,1$, and $M_t:=\gamma_{\alpha_0U_0\to \alpha_1U_1}(t)$, then
	\[
	\gamma_{K_0\to K_1}(t)=V_\ast\ot M_t=\alpha_tV_\ast\ot U_t\in\mathcal G(V_\ast),
	\]
	where $\alpha_t:=(\det M_t)^{1/n}$ and $U_t:=(\det M_t)^{-1/n}M_t$.
\end{lemma}

\begin{proof}
	For $\mathcal F(U_\ast)$, put
	\[
	B:=(V_0^{1/2}V_1V_0^{1/2})^{1/2},
	\qquad
	S:=V_0^{-1/2}BV_0^{-1/2}\in\Sp^n.
	\]
	The corresponding Bures transport on $\Sp^{n^2}$ is
	\[
	T:=K_0^{-1/2}\bigl(K_0^{1/2}K_1K_0^{1/2}\bigr)^{1/2}K_0^{-1/2}
	=
	S\ot I_n.
	\]
	Substitution in \eqref{eq:BuresGeoGeneral} yields $\gamma_{K_0\to K_1}(t)=\gamma_{V_0\to V_1}(t)\ot U_\ast$.
	
	For $\mathcal G(V_\ast)$, write $K_i=V_\ast\ot M_i$, where $M_i:=\alpha_iU_i\in\Sp^n$.
	Repeating the calculation in the second factor gives
	\[
	T:=K_0^{-1/2}\bigl(K_0^{1/2}K_1K_0^{1/2}\bigr)^{1/2}K_0^{-1/2}
	=
	I_n\ot S,
	\]
	where $S=M_0^{-1/2}(M_0^{1/2}M_1M_0^{1/2})^{1/2}M_0^{-1/2}$. Then $\gamma_{K_0\to K_1}(t)=V_\ast\ot\gamma_{M_0\to M_1}(t)=V_\ast\ot M_t$. Since $M_t\in\Sp^n$, setting $\alpha_t:=(\det M_t)^{1/n}$ and $U_t:=(\det M_t)^{-1/n}M_t$ produces $\det U_t=1$ and $V_\ast\ot M_t=\alpha_tV_\ast\ot U_t$.
\end{proof}

\subsection{Fixed-chart geodesic closure}
The first step is to work in a prescribed simultaneous diagonalizing chart. The basis is fixed throughout the statement; the endpoint rigidity theorem below gives the corresponding endpoint-pair closure conclusion without this commuting hypothesis.

\begin{theorem}
	\label{thm:closurecriterion}
	Let $K_i=V_i\ot U_i$ with $(U_i,V_i)\in\mcM_n$, $i=0,1$, so that $\det U_i=1$, and assume that $U_0,U_1$ are simultaneously diagonalizable and $V_0,V_1$ are also simultaneously diagonalizable. Write
	\[
	U_i=Q\diag(u_{i,1},\dots,u_{i,n})Q^\top,
	\qquad
	V_i=R\diag(v_{i,1},\dots,v_{i,n})R^\top.
	\]
	Let $K_t$ denote the Bures geodesic from $K_0$ to $K_1$. In the common basis $R\ot Q$,
	\[
	K_t=(R\ot Q)\diag\bigl(h_{pq}(t)^2\bigr)_{q,p=1}^n(R\ot Q)^\top,
	\]
	where
	\[
	h_{pq}(t)=(1-t)\sqrt{u_{0,p}v_{0,q}}+t\sqrt{u_{1,p}v_{1,q}}.
	\]
	Set $H_t=[h_{pq}(t)]_{p,q=1}^n$.
	Under the above simultaneous-diagonalization hypothesis, the following fixed-chart conditions are equivalent:
	\begin{enumerate}[label=(\roman*),leftmargin=2.4em]
		\item for some $t\in(0,1)$ one has $K_t\in\mcK_n$;
		\item for every $t\in[0,1]$ one has $K_t\in\mcK_n$;
		\item for some $t\in(0,1)$, the square-root profile $H_t$ has rank one;
		\item the endpoints lie on a common factor leaf, namely either $K_0,K_1\in\mathcal F(U_0)=\mathcal F(U_1)$ or $K_0,K_1\in\mathcal G(V_0)=\mathcal G(V_1)$.
	\end{enumerate}
	Equivalently, within this fixed commuting chart, the commuting Bures geodesic remains in $\mcK_n$ precisely when the endpoints move along a single factor leaf.
\end{theorem}

\begin{proof}
	Since
	\[
	K_i=(R\ot Q)\diag(v_{i,q}u_{i,p})_{q,p=1}^n(R\ot Q)^\top,
	\qquad i=0,1,
	\]
	we may work throughout in the common basis $R\ot Q$, in which both $K_0$ and $K_1$ are diagonal. Since $K_0$ and $K_1$ commute, \eqref{eq:BuresGeoCommuting} implies
	\[
	K_t=\bigl((1-t)K_0^{1/2}+tK_1^{1/2}\bigr)^2,
	\]
	which is diagonal with entries $h_{pq}(t)^2$.
	
	Set $a_p:=\sqrt{u_{0,p}}$, $c_p:=\sqrt{u_{1,p}}$, $b_q:=\sqrt{v_{0,q}}$, and $d_q:=\sqrt{v_{1,q}}$. Define $H_t:=[h_{pq}(t)]_{p,q}=(1-t)ab^\top+tcd^\top$.
	For $t\in(0,1)$, all entries of $H_t$ are positive. We claim that
	\[
	K_t\in\mcK_n
	\quad\Longleftrightarrow\quad
	[h_{pq}(t)^2]_{p,q}\text{ has rank one}.
	\]
	In the basis $R\ot Q$ the matrix $K_t$ is diagonal. If $K_t\in\mcK_n$, then $K_t=\widetilde V\ot \widetilde U$ for some $(\widetilde U,\widetilde V)\in\mcM_n$. Conjugating by $R\ot Q$ yields $(R^\top\widetilde V R)\ot(Q^\top\widetilde UQ)$, which is diagonal. Since both factors are positive definite, each has positive diagonal entries. A Kronecker product of the two factors can be diagonal only when both factors are diagonal: any off-diagonal entry in one factor, paired with a positive diagonal entry in the other factor, produces an off-diagonal entry in the product. Thus
	\[
	R^\top\widetilde V R=\diag(y_1,\dots,y_n),
	\qquad
	Q^\top\widetilde UQ=\diag(x_1,\dots,x_n),
	\]
	with $x_p,y_q>0$, and then $[h_{pq}(t)^2]_{p,q}=[x_p y_q]_{p,q}=xy^\top$, which has rank one. Conversely, if $[h_{pq}(t)^2]=xy^\top$ with $x,y\in\R_{++}^n$, define $U:=Q\diag(x)Q^\top$ and $V:=R\diag(y)R^\top$. Then $K_t=(R\ot Q)\diag(y_qx_p)_{q,p}(R\ot Q)^\top=V\ot U$. The representation may fail to satisfy $\det U=1$, but the Kronecker product is unchanged under reciprocal rescaling. With $c:=(\det U)^{-1/n}>0$, $\widetilde U:=cU$, and $\widetilde V:=c^{-1}V$, we obtain $(\widetilde U,\widetilde V)\in\mcM_n$ and $\widetilde V\ot\widetilde U=K_t$, so that $K_t\in\mcK_n$.
	
	Since all entries are positive, the rank-one property of $[h_{pq}(t)^2]$ is equivalent to the rank-one property of $H_t$ itself: if $H_t=rs^\top$, then $[h_{pq}(t)^2]=(r\circ r)(s\circ s)^\top$. Conversely, if $[h_{pq}(t)^2]=xy^\top$ with $x,y\in\R_{++}^n$, then
	$H_t=[\sqrt{x_py_q}]_{p,q}=(\sqrt{x})(\sqrt{y})^\top$, since $H_t$ has positive entries. Hence \textup{(i)} and \textup{(iii)} are equivalent.
	
	The matrix $H_t$ is a sum of two rank-one matrices. Since $t\in(0,1)$ and $a,b,c,d$ are nonzero positive vectors, this sum has rank one only in the two standard alternatives: the left vectors are linearly dependent or the right vectors are linearly dependent. Thus $a\parallel c$ or $b\parallel d$.
	
	If $a\parallel c$, write $c=\tau a$ with $\tau>0$. Squaring componentwise gives $u_{1,p}=\tau^2u_{0,p}$ for every $p$, so $U_1=\tau^2U_0$. The determinant normalization gives $\tau^{2n}=1$, hence $\tau=1$ and $U_1=U_0$. The endpoints then lie on the common leaf $\mathcal F(U_0)=\mathcal F(U_1)$. If instead $b\parallel d$, then $d=\rho b$ for some $\rho>0$, and consequently $V_1=\rho^2V_0$. In that case the endpoints lie on the common $V$-factor leaf $\mathcal G(V_0)=\mathcal G(V_1)$. We have proved \textup{(iii)}$\Rightarrow$\textup{(iv)}.
	
	Conversely, assume \textup{(iv)}. If $K_0$ and $K_1$ lie on a common normalized $U$-leaf, then $U_1=U_0$ and $a=c$, so
	\[
	H_t=a\bigl((1-t)b+td\bigr)^\top
	\]
	has rank one for every $t\in[0,1]$. If they lie on a common normalized $V$-leaf, then $V_1=\mu V_0$ for some $\mu>0$, hence $d=\sqrt{\mu}\,b$, and
	\[
	H_t=\bigl((1-t)a+t\sqrt{\mu}\,c\bigr)b^\top
	\]
	again has rank one for every $t\in[0,1]$. The rank-one characterization gives $K_t\in\mcK_n$ for all $t$, proving \textup{(iv)}$\Rightarrow$\textup{(ii)}. The implication \textup{(ii)}$\Rightarrow$\textup{(i)} is immediate.
\end{proof}

\begin{remark}
	\label{rem:commutingimmediatedeparture}
	Under the hypotheses of \Cref{thm:closurecriterion}, suppose that neither factor pair is proportional: $U_1\ne \lambda U_0$ for every $\lambda>0$, and $V_1\ne \mu V_0$ for every $\mu>0$. Then $K_t\notin \mcK_n$ for $0<t<1$. Any interior return to $\mcK_n$ would force the endpoints onto a common factor leaf by \Cref{thm:closurecriterion}, contradicting the two assumptions above.
\end{remark}

\begin{remark}
	\label{rem:closurebasis}
	The simultaneous diagonalizing bases $Q$ and $R$ need not be unique when repeated eigenvalues occur. This does not affect \Cref{thm:closurecriterion}: conditions \textup{(i)}, \textup{(ii)}, and \textup{(iv)} are intrinsic statements about $K_0$, $K_1$, and the ambient geodesic, while the rank-one condition in \textup{(iii)} is equivalent to them within the chosen chart. Only the profile matrices $H_t$ and $M_t$ introduced below depend on that choice.
\end{remark}

\begin{example}
	\label{ex:twodimnonclosure}
	Let $n=2$ and take diagonal factors
	\[
	U_0=I_2,\qquad V_0=I_2,\qquad
	U_1=\diag(2,1/2),\qquad V_1=\diag(3,1).
	\]
	The endpoints commute and satisfy $\det U_0=\det U_1=1$, but they do not lie on a common factor leaf: $U_1\ne U_0$ and $V_1$ is not a scalar multiple of $V_0$. Hence \Cref{thm:closurecriterion} implies that the Bures geodesic is outside $\mcK_2$ for every $t\in(0,1)$. Equivalently, the square-root profile
	\[
	H_t=(1-t)\begin{pmatrix}1&1\\1&1\end{pmatrix}
	+t\begin{pmatrix}\sqrt6&\sqrt2\\ \sqrt{3/2}&1/\sqrt2\end{pmatrix}
	\]
	has rank two for every interior time.
	\[
	\det H_t
	=t(1-t)\left(\sqrt6+\frac1{\sqrt2}-\sqrt2-\sqrt{\frac32}\right)>0,
	\qquad 0<t<1.
	\]
\end{example}
\subsection{Departure moduli}

The fixed-chart criterion is qualitative, but it also gives numerical diagnostics. Two Frobenius departure moduli are used, one for the square-root profile and one for the diagonal profile. They measure failure of the rank-one condition in \Cref{thm:closurecriterion}; they are not ambient distances to $\mcK_n$. The relevant singular-value computations have dimension at most three.

\begin{proposition}
	\label{prop:departuremodulus}
	Assume the hypotheses and notation of \cref{thm:closurecriterion}. Define $a:=(\sqrt{u_{0,1}},\dots,\sqrt{u_{0,n}})^\top$, $c:=(\sqrt{u_{1,1}},\dots,\sqrt{u_{1,n}})^\top$, $b:=(\sqrt{v_{0,1}},\dots,\sqrt{v_{0,n}})^\top$, and $d:=(\sqrt{v_{1,1}},\dots,\sqrt{v_{1,n}})^\top$. For $t\in[0,1]$, let $H_t:=(1-t)\,ab^\top+t\,cd^\top\in\R^{n\times n}$, so that the diagonal entries of the commuting geodesic $K_t$ in the basis $R\ot Q$ are given by the entrywise squares of $H_t$.
	
	Define the square-root geodesic departure modulus
	\begin{equation}
		\label{eq:departuremodulusdef}
		\delta_{\mathrm{geo}}(t)
		:=
		\inf_{x\in\R_{++}^n,\ y\in\R_{++}^n}
		\norm{H_t-xy^\top}_F .
	\end{equation}
	\begin{enumerate}
		\item[(i)] One has
		\begin{equation}
			\label{eq:departuremodulussigma}
			\delta_{\mathrm{geo}}(t)=\sigma_2(H_t),
		\end{equation}
		where $\sigma_2(H_t)$ is the second singular value of $H_t$.
		
		\item[(ii)] Set $A:=\norm{a}^2$, $C:=\norm{c}^2$, $\rho:=\ip{a}{c}$, $B:=\norm{b}^2$, $D:=\norm{d}^2$, and $\sigma:=\ip{b}{d}$. Then
		\begin{equation}
			\label{eq:departuremodulusclosedform}
			\delta_{\mathrm{geo}}(t)^2
			=
			\frac12\Bigl(T(t)-\sqrt{T(t)^2-4\Delta(t)}\Bigr),
		\end{equation}
		where
		\begin{equation}
			\label{eq:TtDeparture}
			T(t)
			=
			(1-t)^2AB+2t(1-t)\rho\sigma+t^2CD
		\end{equation}
		and
		\begin{equation}
			\label{eq:DeltatDeparture}
			\Delta(t)
			=
			t^2(1-t)^2\,(AC-\rho^2)(BD-\sigma^2).
		\end{equation}
		
		\item[(iii)] For $t\in(0,1)$, $\delta_{\mathrm{geo}}(t)=0$ if and only if $K_0$ and $K_1$ lie on a common factor leaf. In particular, if the endpoints lie on no common factor leaf, then $\delta_{\mathrm{geo}}(t)>0$ for all $t\in(0,1)$.
		
		\item[(iv)] As $t\to0$,
		\begin{equation}
			\label{eq:departuremodulusasymptotic}
			\delta_{\mathrm{geo}}(t)^2
			=
			t^2\,\frac{(AC-\rho^2)(BD-\sigma^2)}{AB}
			+O(t^3).
		\end{equation}
	\end{enumerate}
	
	For the diagonal profile $M_t:=[h_{pq}(t)^2]_{p,q=1}^n=H_t\circ H_t$, define
	\begin{equation}
		\label{eq:covdepartdef}
		\delta_{\mathrm{diag}}(t)
		:=
		\inf_{x\in\R_{++}^n,\ y\in\R_{++}^n}
		\norm{M_t-xy^\top}_F .
	\end{equation}
	Then
	\begin{equation}
		\label{eq:covdepartsigma}
		\delta_{\mathrm{diag}}(t)^2=\sum_{j\ge 2}\sigma_j(M_t)^2
		=\sigma_2(M_t)^2+\sigma_3(M_t)^2,
	\end{equation}
	as $\operatorname{rank}(M_t)\le 3$. More explicitly, with
	\[
	p_1:=a\circ a,\quad p_2:=a\circ c,\quad p_3:=c\circ c,
	\qquad
	q_1:=b\circ b,\quad q_2:=b\circ d,\quad q_3:=d\circ d,
	\]
	\[
	P_t:=\bigl[(1-t)p_1,\ \sqrt{2t(1-t)}\,p_2,\ t p_3\bigr],
	\qquad
	Q_t:=\bigl[(1-t)q_1,\ \sqrt{2t(1-t)}\,q_2,\ t q_3\bigr],
	\]
	one has
	\begin{equation}
		\label{eq:Mtfactorization}
		M_t=P_tQ_t^\top .
	\end{equation}
	If $\lambda_1(t)\ge\lambda_2(t)\ge\lambda_3(t)\ge0$ are the eigenvalues of
	\begin{equation}
		\label{eq:Gram3x3}
		\Gamma_t:=(P_t^\top P_t)^{1/2}(Q_t^\top Q_t)(P_t^\top P_t)^{1/2},
	\end{equation}
	then
	\begin{equation}
		\label{eq:covdepart3x3}
		\delta_{\mathrm{diag}}(t)^2=\lambda_2(t)+\lambda_3(t)=\tr(\Gamma_t)-\lambda_1(t).
	\end{equation}
	Moreover, for $t\in(0,1)$,
	\begin{equation}
		\label{eq:covdepartzero}
		\delta_{\mathrm{diag}}(t)=0
		\quad\Longleftrightarrow\quad
		K_0\text{ and }K_1\text{ lie on a common factor leaf}.
	\end{equation}
\end{proposition}

\begin{proof}
	The matrix $H_t=(1-t)\,ab^\top+t\,cd^\top$ has positive entries and rank at most two. The Eckart--Young theorem \cite{EckartYoung1936} yields the unrestricted Frobenius error $\sigma_2(H_t)$. Since $H_tH_t^\top$ and $H_t^\top H_t$ are entrywise positive, Perron--Frobenius allows the leading singular vectors to be chosen positive. The optimal rank-one approximant is therefore admissible in \eqref{eq:departuremodulusdef}, which proves \eqref{eq:departuremodulussigma}.
	
	For the closed form, write
	\[
	H_t=U_tV_t^\top,
	\qquad
	U_t:=\bigl[(1-t)a,\ tc\bigr]\in\R^{n\times 2},
	\qquad
	V_t:=\bigl[b,\ d\bigr]\in\R^{n\times 2}.
	\]
	The nonzero squared singular values are the nonzero eigenvalues of $H_t^\top H_t=V_t(U_t^\top U_t)V_t^\top$, hence of $(U_t^\top U_t)(V_t^\top V_t)$. Equivalently, they are the eigenvalues of the symmetric positive semidefinite matrix $(U_t^\top U_t)^{1/2}(V_t^\top V_t)(U_t^\top U_t)^{1/2}$, which has the same nonzero spectrum. Here
	\[
	U_t^\top U_t
	=
	\begin{pmatrix}
		(1-t)^2A & t(1-t)\rho \\
		t(1-t)\rho & t^2C
	\end{pmatrix},
	\qquad
	V_t^\top V_t
	=
	\begin{pmatrix}
		B & \sigma\\
		\sigma & D
	\end{pmatrix}.
	\]
	The trace and determinant of this $2\times2$ product are
	\[
	\operatorname{tr}\bigl((U_t^\top U_t)(V_t^\top V_t)\bigr)
	=
	(1-t)^2AB+2t(1-t)\rho\sigma+t^2CD,
	\]
	and
	\[
	\det(U_t^\top U_t)\det(V_t^\top V_t)
	=
	t^2(1-t)^2(AC-\rho^2)(BD-\sigma^2),
	\]
	namely $T(t)$ and $\Delta(t)$. The smaller eigenvalue yields \eqref{eq:departuremodulusclosedform}. Since $T(t)=AB+O(t)$ and
	\[
	\Delta(t)=t^2(AC-\rho^2)(BD-\sigma^2)+O(t^3),
	\]
	the expansion
	\[
	\frac12\left(T(t)-\sqrt{T(t)^2-4\Delta(t)}\right)
	=\frac{\Delta(t)}{T(t)}+O(t^4)
	\]
	proves \eqref{eq:departuremodulusasymptotic}.
	
	For $t\in(0,1)$, $\delta_{\mathrm{geo}}(t)=0$ is equivalent to $\Delta(t)=0$, i.e., $AC-\rho^2=0$ or $BD-\sigma^2=0$. By equality in Cauchy--Schwarz, this means that $a,c$ are collinear or $b,d$ are collinear, which is equivalent to $U_1=\lambda U_0$ or $V_1=\mu V_0$ for positive scalars. Since $\det U_0=\det U_1=1$, the first alternative gives $\lambda=1$. Thus this is exactly the factor-leaf condition in \Cref{thm:closurecriterion}, and proves the vanishing statement.
	
	For the diagonal profile,
	\[
	M_t=[h_{pq}(t)^2]_{p,q}
	=(1-t)^2(a\circ a)(b\circ b)^\top
	+2t(1-t)(a\circ c)(b\circ d)^\top
	+t^2(c\circ c)(d\circ d)^\top.
	\]
	Equation \eqref{eq:Mtfactorization} gives $\operatorname{rank}(M_t)\le3$. The same Eckart--Young/Perron--Frobenius argument gives \eqref{eq:covdepartsigma}. Set $A_t=P_t^\top P_t$ and $B_t=Q_t^\top Q_t$. From $M_t=P_tQ_t^\top$, the nonzero eigenvalues of $M_t^\top M_t$ coincide with those of $Q_tA_tQ_t^\top$, hence also with those of $A_tB_t$. These are the eigenvalues of $\Gamma_t=A_t^{1/2}B_tA_t^{1/2}$. For singular $A_t$, apply it to $A_t+\varepsilon I_3$ and let $\varepsilon\downarrow0$. This proves \eqref{eq:covdepart3x3}.
	
	Finally, $\delta_{\mathrm{diag}}(t)=0$ is equivalent to $\operatorname{rank}(M_t)=1$. With positive entries, this is equivalent to $\operatorname{rank}(H_t)=1$: if $M_t=xy^\top$, then
	$H_t=[\sqrt{x_p y_q}]_{p,q}=(\sqrt{x})(\sqrt{y})^\top$, and the converse follows from $M_t=H_t\circ H_t$. \Cref{thm:closurecriterion} completes the proof of \eqref{eq:covdepartzero}.
\end{proof}

\begin{remark}
	\label{rem:departuremodulus}
	The moduli $\delta_{\mathrm{geo}}(t)$ and $\delta_{\mathrm{diag}}(t)$ are fixed-chart departure moduli, not ambient distances to $\mcK_n$. Their tractability comes from the bounds $\operatorname{rank}(H_t)\le 2$ and $\operatorname{rank}(H_t\circ H_t)\le 3$, and both vanish precisely in the factor-leaf cases of \Cref{thm:closurecriterion}.
\end{remark}

\subsection{Endpoint tangency and rigidity}
The fixed-chart closure theorem gives an explicit rank-one description under simultaneous diagonalization. For arbitrary Kronecker endpoints, the endpoint calculation is more effective than a diagnostic alone: the partial-trace residual first gives a tangency test, and then the endpoint structure forces a factor-leaf conclusion.

\begin{lemma}
	\label{lem:partialtraceresidual}
	For $Z\in\Sym^{n^2}$, define
	\begin{equation}
		\label{eq:Pioperator}
		\Pi(Z)
		:=
		Z-\frac1n I_n\ot \operatorname{tr}_1(Z)-\frac1n \operatorname{tr}_2(Z)\ot I_n
		+\frac{\tr(Z)}{n^2}I_{n^2}.
	\end{equation}
	Let
	\[
	\mathcal T_{\otimes}:=\{I_n\ot A+B\ot I_n:A,B\in\Sym^n,\ \tr(A)=0\}.
	\]
	Then $Z-\Pi(Z)\in\mathcal T_{\otimes}$ and $\Pi(Z)$ is Frobenius-orthogonal to $\mathcal T_{\otimes}$. In particular,
	\begin{equation}
		\label{eq:Pikernel}
		\ker\Pi=\mathcal T_{\otimes}.
	\end{equation}
	More explicitly, $\Pi(Z)=0$ if and only if
	\[
	Z=I_n\ot A+B\ot I_n,
	\qquad
	A=\frac1n\operatorname{tr}_1(Z)-\frac{\tr(Z)}{n^2}I_n,\quad
	B=\frac1n\operatorname{tr}_2(Z),
	\]
	with $\tr(A)=0$.
\end{lemma}

\begin{proof}
	If $Z=I_n\ot A+B\ot I_n$ and $\tr(A)=0$, then
	\[
	\operatorname{tr}_1(Z)=nA+\tr(B)I_n,
	\qquad
	\operatorname{tr}_2(Z)=nB,
	\qquad
	\tr(Z)=n\,\tr(B),
	\]
	so $\Pi(Z)=0$. Conversely, $\Pi(Z)=0$ implies
	\[
	Z
	=
	\frac1n I_n\ot \operatorname{tr}_1(Z)+\frac1n \operatorname{tr}_2(Z)\ot I_n
	-\frac{\tr(Z)}{n^2}I_{n^2}
	=
	I_n\ot A+B\ot I_n,
	\]
	where
	\[
	A:=\frac1n\operatorname{tr}_1(Z)-\frac{\tr(Z)}{n^2}I_n,
	\qquad
	B:=\frac1n\operatorname{tr}_2(Z).
	\]
	Here $\tr(A)=0$ because $\tr(\operatorname{tr}_1(Z))=\tr(Z)$. Thus $\ker\Pi=\mathcal T_{\otimes}$.
	
	For general $Z$, the displayed formulas give $Z-\Pi(Z)=I_n\ot A+B\ot I_n\in\mathcal T_{\otimes}$. Moreover,
	\[
	\operatorname{tr}_1(\Pi(Z))=0,
	\qquad
	\operatorname{tr}_2(\Pi(Z))=0.
	\]
	Hence, for every $I_n\ot A+B\ot I_n\in\mathcal T_{\otimes}$,
	\[
	\ip{\Pi(Z)}{I_n\ot A+B\ot I_n}_F
	=
	\tr\bigl(A\,\operatorname{tr}_1(\Pi(Z))\bigr)
	+
	\tr\bigl(B\,\operatorname{tr}_2(\Pi(Z))\bigr)
	=0.
	\]
	This establishes the orthogonality statement.
\end{proof}

\begin{theorem}
	\label{thm:localobstruction}
	Let $K_i=V_i\ot U_i\in\mcK_n$, $i=0,1$, and let $\gamma_{K_0\to K_1}$ denote the ambient Bures geodesic. Set
	\[
	T_{0\to1}:=
	K_0^{-1/2}\bigl(K_0^{1/2}K_1K_0^{1/2}\bigr)^{1/2}K_0^{-1/2}.
	\]
	If there exists $\varepsilon>0$ such that $\gamma_{K_0\to K_1}([0,\varepsilon))\subset\mcK_n$, then
	\begin{equation}
		\label{eq:initialvelocitycriterion}
		\dot\gamma_{K_0\to K_1}(0)
		=
		\bigl(T_{0\to1}-I_{n^2}\bigr)K_0+K_0\bigl(T_{0\to1}-I_{n^2}\bigr)
		\in T_{K_0}\mcK_n.
	\end{equation}
	This tangency condition is equivalent to the following whitened Kronecker-sum condition: with
	\begin{equation}
		\label{eq:Z0definition}
		Z_0
		:=
		\bigl(V_0^{-1/2}\ot U_0^{-1/2}\bigr)\,
		\dot\gamma_{K_0\to K_1}(0)\,
		\bigl(V_0^{-1/2}\ot U_0^{-1/2}\bigr),
	\end{equation}
	there exist $A,B\in\Sym^n$ with $\tr(A)=0$ such that
	\begin{equation}
		\label{eq:whitenedkroneckersum}
		Z_0=I_n\ot A+B\ot I_n.
	\end{equation}
	Equivalently,
	\begin{equation}
		\label{eq:Piobstruction}
		\Pi(Z_0)=0.
	\end{equation}
	Hence, if $\Pi(Z_0)\neq 0$, then the ambient Bures geodesic does not remain in $\mcK_n$ on any nontrivial initial interval starting from $K_0$.
\end{theorem}

\begin{proof}
	If $\gamma_{K_0\to K_1}([0,\varepsilon))\subset\mcK_n$, its initial velocity lies in $T_{K_0}\mcK_n$. Differentiating
	\[
	\gamma_{K_0\to K_1}(t)
	=
	\bigl((1-t)I_{n^2}+tT_{0\to1}\bigr)K_0\bigl((1-t)I_{n^2}+tT_{0\to1}\bigr)
	\]
	at $t=0$ yields \eqref{eq:initialvelocitycriterion}.
	
	The tangent-space formula in \Cref{lem:modelmanifold} gives
	\[
	\dot\gamma_{K_0\to K_1}(0)=V_0\ot H_U+H_V\ot U_0,
	\qquad
	\tr(U_0^{-1}H_U)=0.
	\]
	Whitening by $V_0^{-1/2}\ot U_0^{-1/2}$ gives
	\[
	Z_0
	=
	I_n\ot(U_0^{-1/2}H_UU_0^{-1/2})
	+
	(V_0^{-1/2}H_VV_0^{-1/2})\ot I_n.
	\]
	Thus \eqref{eq:whitenedkroneckersum} holds with
	\[
	A:=U_0^{-1/2}H_UU_0^{-1/2},
	\qquad
	B:=V_0^{-1/2}H_VV_0^{-1/2},
	\]
	and $\tr(A)=0$. Reversing the whitening proves the converse implication, so \eqref{eq:whitenedkroneckersum} is precisely the whitened form of tangency. By \Cref{lem:partialtraceresidual}, \eqref{eq:whitenedkroneckersum} is equivalent to \eqref{eq:Piobstruction}. The final assertion follows by contraposition.
\end{proof}

\begin{lemma}
	\label{lem:whitenedendpointformula}
	Let $K_i=V_i\ot U_i\in\mcK_n$, $i=0,1$, and define the factor Bures transports
	\[
	S_V:=V_0^{-1/2}(V_0^{1/2}V_1V_0^{1/2})^{1/2}V_0^{-1/2},
	\qquad
	S_U:=U_0^{-1/2}(U_0^{1/2}U_1U_0^{1/2})^{1/2}U_0^{-1/2}.
	\]
	Set
	\[
	P:=V_0^{-1/2}S_VV_0^{1/2},
	\qquad
	Q:=U_0^{-1/2}S_UU_0^{1/2}.
	\]
	Then the ambient transport factorizes as $T_{0\to1}=S_V\ot S_U$, and the whitened initial velocity in \Cref{thm:localobstruction} is
	\begin{equation}
		\label{eq:explicitZ0}
		Z_0=P\ot Q+P^\top\ot Q^\top-2I_{n^2}.
	\end{equation}
	Thus the endpoint residual $\Pi(Z_0)$ is computable from the two factor-size transports; $\norm{\Pi(Z_0)}_F=0$ is precisely the endpoint tangency condition.
\end{lemma}

\begin{proof}
	The Kronecker product square-root identity gives
	\[
	K_0^{1/2}K_1K_0^{1/2}
	=
	(V_0^{1/2}V_1V_0^{1/2})\ot(U_0^{1/2}U_1U_0^{1/2}),
	\]
	and hence $T_{0\to1}=S_V\ot S_U$. Using
	\[
	\dot\gamma_{K_0\to K_1}(0)
	=
	(T_{0\to1}-I_{n^2})K_0+K_0(T_{0\to1}-I_{n^2}),
	\]
	and whitening by $V_0^{-1/2}\ot U_0^{-1/2}$ gives
	\[
	Z_0
	=
	(V_0^{-1/2}S_VV_0^{1/2})\ot(U_0^{-1/2}S_UU_0^{1/2})
	+
	(V_0^{1/2}S_VV_0^{-1/2})\ot(U_0^{1/2}S_UU_0^{-1/2})
	-2I_{n^2}.
	\]
	Since $S_V$ and $S_U$ are symmetric, the second Kronecker product is $P^\top\ot Q^\top$, proving \eqref{eq:explicitZ0}. The final claim follows from \Cref{thm:localobstruction}.
\end{proof}

\begin{lemma}
	\label{lem:transportkroneckersumrigidity}
	Let $D_V=\diag(v_1,\dots,v_n)$ and $D_U=\diag(u_1,\dots,u_n)$ be positive diagonal matrices, let $S_V,S_U\in\Sp^n$, and set
	\[
	P:=D_V^{-1/2}S_VD_V^{1/2},
	\qquad
	Q:=D_U^{-1/2}S_UD_U^{1/2}.
	\]
	If
	\begin{equation}
		\label{eq:transportkroneckersum}
		P\ot Q+P^\top\ot Q^\top-2I_{n^2}=I_n\ot A+B\ot I_n
	\end{equation}
	for some $A,B\in\Sym^n$, then either $P=\lambda I_n$ or $Q=\mu I_n$ for some positive scalar.
\end{lemma}

\begin{proof}
	The proof begins with the sign information inherited from the symmetric positive-definite transports. Because $S_V$ and $S_U$ are symmetric,
	\[
	p_{ij}=s^V_{ij}\sqrt{v_j/v_i},\qquad
	p_{ji}=s^V_{ij}\sqrt{v_i/v_j},
	\]
	and similarly
	\[
	q_{k\ell}=s^U_{k\ell}\sqrt{u_\ell/u_k},\qquad
	q_{\ell k}=s^U_{k\ell}\sqrt{u_k/u_\ell}.
	\]
	It follows that each pair $p_{ij},p_{ji}$, and likewise $q_{k\ell},q_{\ell k}$, has the same weak sign and the two entries vanish simultaneously. Also, $P$ and $Q$ are similar to $S_V$ and $S_U$, respectively, so a scalar value of either matrix is necessarily positive.
	
	Suppose first that $P$ has a nonzero off-diagonal entry $p_{ij}$, $i\ne j$. The $(i,j)$ block of \eqref{eq:transportkroneckersum} gives
	\[
	p_{ij}Q+p_{ji}Q^\top=b_{ij}I_n .
	\]
	For $k\ne\ell$, the $(k,\ell)$ entry of this block identity is
	\[
	p_{ij}q_{k\ell}+p_{ji}q_{\ell k}=0 .
	\]
	The two summands have the same weak sign: $p_{ij}$ and $p_{ji}$ have the same nonzero sign, and $q_{k\ell}$ and $q_{\ell k}$ have the same weak sign. The sign agreement forces both summands to vanish. As $p_{ij}$ and $p_{ji}$ are nonzero, $q_{k\ell}=q_{\ell k}=0$ for every $k\ne\ell$, and $Q$ is diagonal. The diagonal entries in the same block identity satisfy
	\[
	(p_{ij}+p_{ji})q_{kk}=b_{ij},\qquad k=1,\dots,n.
	\]
	Here $p_{ij}+p_{ji}\ne0$, again by the common sign, so all $q_{kk}$ are equal and $Q=\mu I_n$ with $\mu>0$.
	
	It remains to consider the case where $P$ is diagonal, say $P=\diag(p_1,\dots,p_n)$. If $Q$ has a nonzero off-diagonal entry, then the off-diagonal $(k,\ell)$ entries of the diagonal block equations give
	\[
	p_i(q_{k\ell}+q_{\ell k})=a_{k\ell},\qquad i=1,\dots,n.
	\]
	The paired-sign property gives $q_{k\ell}+q_{\ell k}\ne0$. Since the right-hand side is independent of $i$, all $p_i$ are equal, and $P=\lambda I_n$ with $\lambda>0$.
	
	Finally, suppose that both $P=\diag(p_1,\dots,p_n)$ and $Q=\diag(q_1,\dots,q_n)$ are diagonal. The diagonal entries of the left side of \eqref{eq:transportkroneckersum} are $2(p_iq_k-1)$. Since the diagonal entries of $I_n\ot A+B\ot I_n$ have vanishing mixed differences,
	\[
	0=2(p_iq_k-1)-2(p_iq_\ell-1)-2(p_jq_k-1)+2(p_jq_\ell-1)
	=2(p_i-p_j)(q_k-q_\ell).
	\]
	Thus either all $p_i$ are equal or all $q_k$ are equal. Again the scalar is positive because $P$ and $Q$ are similar to positive definite matrices.
\end{proof}

\begin{theorem}
	\label{thm:endpointrigidity}
	Let $K_i=V_i\ot U_i\in\mcK_n$, $i=0,1$, with $\det U_0=\det U_1=1$, and let $Z_0$ be the whitened initial velocity in \Cref{thm:localobstruction}. The following are equivalent:
	\begin{enumerate}[label=(\roman*),leftmargin=2.4em]
		\item for some $\varepsilon>0$, $\gamma_{K_0\to K_1}([0,\varepsilon))\subset\mcK_n$;
		\item $\Pi(Z_0)=0$;
		\item $K_0$ and $K_1$ lie on a common factor leaf, that is, either $U_1=U_0$ or $V_1$ is a positive scalar multiple of $V_0$;
		\item $\gamma_{K_0\to K_1}([0,1])\subset\mcK_n$.
	\end{enumerate}
	Equivalently, for Kronecker endpoint pairs, endpoint-local Bures geodesic closure and whole-segment Bures geodesic closure occur precisely in the factor-leaf cases.
\end{theorem}

\begin{proof}
	The implication \textup{(i)}$\Rightarrow$\textup{(ii)} is \Cref{thm:localobstruction}, and \textup{(iii)}$\Rightarrow$\textup{(iv)} follows from \Cref{lem:geodesicallyclosedleaves}. The implication \textup{(iv)}$\Rightarrow$\textup{(i)} is immediate. It remains to prove \textup{(ii)}$\Rightarrow$\textup{(iii)}.
	
	The condition $\Pi(Z_0)=0$ is invariant under factorwise orthogonal changes of basis. For $\widetilde Z=(R^\top\ot Q^\top)Z(R\ot Q)$ with $R,Q\in O(n)$, the partial traces transform as $\operatorname{tr}_1(\widetilde Z)=Q^\top\operatorname{tr}_1(Z)Q$ and $\operatorname{tr}_2(\widetilde Z)=R^\top\operatorname{tr}_2(Z)R$, while $\tr(\widetilde Z)=\tr(Z)$. Thus $\Pi(\widetilde Z)=(R^\top\ot Q^\top)\Pi(Z)(R\ot Q)$. For the correspondingly transformed endpoints, the whitened velocity is precisely $(R^\top\ot Q^\top)Z_0(R\ot Q)$. It is therefore enough to assume that
	$U_0=\diag(u_1,\dots,u_n)$ and $V_0=\diag(v_1,\dots,v_n)$.
	
	Let $S_V,S_U,P,Q$ be as in \Cref{lem:whitenedendpointformula}. By \Cref{lem:partialtraceresidual,lem:whitenedendpointformula}, condition \textup{(ii)} gives
	\[
	P\ot Q+P^\top\ot Q^\top-2I_{n^2}=I_n\ot A+B\ot I_n
	\]
	for some $A,B\in\Sym^n$ with $\tr(A)=0$. The rigidity lemma \Cref{lem:transportkroneckersumrigidity} implies that either $P=\lambda I_n$ or $Q=\mu I_n$, with $\lambda,\mu>0$.
	
	If $P=\lambda I_n$, then $S_V=V_0^{1/2}PV_0^{-1/2}=\lambda I_n$, and 
	\[
	V_1=S_VV_0S_V=\lambda^2V_0.
	\]
	If $Q=\mu I_n$, then $S_U=\mu I_n$ and
	\[
	U_1=S_UU_0S_U=\mu^2U_0.
	\]
	The determinant normalization yields $\mu^{2n}=1$, hence $\mu=1$ and $U_1=U_0$. So the endpoints lie on a common factor leaf, which proves \textup{(iii)}.
\end{proof}

\begin{remark}
	\label{rem:localobstructionrigid}
	\Cref{thm:localobstruction} does not assume simultaneous diagonalization. The admissible whitened tangent space $\{I_n\ot A+B\ot I_n:A,B\in\Sym^n,\ \tr(A)=0\}$ has dimension $n(n+1)-1$, while $\Sym^{n^2}$ has dimension $n^2(n^2+1)/2$. \Cref{thm:endpointrigidity} shows that, among velocities generated by Kronecker endpoints through the ambient Bures geodesic, this first-order condition is rigid rather than merely infinitesimal.
\end{remark}

\begin{corollary}
	\label{cor:2x2obstruction}
	Assume $n=2$, and let $Z_0=[z_{ij}]_{i,j=1}^4$ denote the whitened initial velocity from \Cref{thm:localobstruction}. Then $\Pi(Z_0)=0$ if and only if
	\begin{equation}
		\label{eq:2x2entryconditions}
		z_{14}=z_{23}=0,\qquad
		z_{12}=z_{34},\qquad
		z_{13}=z_{24},\qquad
		z_{11}-z_{22}=z_{33}-z_{44}.
	\end{equation}
	Equivalently, there exist real parameters $\alpha,\beta,\delta,\gamma,\varepsilon$ such that
	\begin{equation}
		\label{eq:2x2pattern}
		Z_0=
		\begin{pmatrix}
			\alpha+\gamma & \beta & \delta & 0\\
			\beta & -\alpha+\gamma & 0 & \delta\\
			\delta & 0 & \alpha+\varepsilon & \beta\\
			0 & \delta & \beta & -\alpha+\varepsilon
		\end{pmatrix}.
	\end{equation}
	Consequently, violation of any one relation in \eqref{eq:2x2entryconditions} rules out local geodesic closure at the endpoint, while satisfaction of all relations forces the endpoints onto a common factor leaf by \Cref{thm:endpointrigidity}.
\end{corollary}

\begin{proof}
	By \Cref{lem:partialtraceresidual}, $\Pi(Z_0)=0$ is equivalent to $Z_0=I_2\ot A+B\ot I_2$ with
	\[
	A=\begin{pmatrix}\alpha&\beta\\ \beta&-\alpha\end{pmatrix},
	\qquad
	B=\begin{pmatrix}\gamma&\delta\\ \delta&\varepsilon\end{pmatrix}.
	\]
	Expanding yields \eqref{eq:2x2pattern}. Conversely, the identities in \eqref{eq:2x2entryconditions} determine such parameters uniquely.
\end{proof}

\begin{example}
	\label{ex:2x2noncommutingexample}
	Consider
	\[
	U_0=\begin{pmatrix}2&0\\0&1/2\end{pmatrix},
	\qquad
	V_0=\begin{pmatrix}4&0\\0&1\end{pmatrix},
	\]
	\[
	U_1=\frac14\begin{pmatrix}5&3\\3&5\end{pmatrix},
	\qquad
	V_1=\begin{pmatrix}1&0\\0&4\end{pmatrix},
	\]
	and set $K_i=V_i\ot U_i\in\mcK_2$. Then $U_0U_1\neq U_1U_0$, so that $K_0$ and $K_1$ fail to commute. Moreover $U_1\neq U_0$ and $V_1$ is not a scalar multiple of $V_0$, so the endpoints lie on no common factor leaf.
	
	Let $T_{0\to1}$ be the Bures transport from $K_0$ to $K_1$, and let $Z_0$ be the whitened initial velocity from \Cref{thm:localobstruction}. A direct calculation gives $z_{12}=15/(4\sqrt{82})\neq 15/\sqrt{82}=z_{34}$, so the $2\times2$ tangency condition from \Cref{cor:2x2obstruction} fails. Equivalently,
	$\Pi(Z_0)\ne0$. Hence the ambient Bures geodesic from $K_0$ to $K_1$ cannot remain in $\mcK_2$ on any initial interval $[0,\varepsilon)$.
\end{example}

\section{Restricted barycenter problems}
\label{sec:barycenters}

Exact barycenter formulas are available in two restricted domains. The first is a fixed commuting coordinate slice, where square-root coordinates convert the determinant-normalized objective into a Rayleigh-quotient maximization for an entrywise positive matrix. The second is a factor leaf, where the problem is globally equivalent, within that leaf, to the standard Bures--Wasserstein barycenter problem on $\Sp^n$. No global formula on all of $\mcK_n$ is asserted here.

For matrices $K_1,\dots,K_N\in\Sp^{n^2}$ and weights $w_i>0$ with $\sum_i w_i=1$, define $\mcJ(K):=\sum_{i=1}^N w_i\,d_{\mathrm B}^2(K,K_i)$. On the determinant-normalized Kronecker model, we write $\mcJ(U,V):=\mcJ(V\ot U)$.

\subsection{Fixed commuting-coordinate slice}
Assume now that all $U_i$ are simultaneously diagonalizable in an orthogonal basis $Q$, and all $V_i$ are simultaneously diagonalizable in an orthogonal basis $R$. Write
\[
U_i=Q\diag(u_{i,1},\dots,u_{i,n})Q^\top,
\qquad
V_i=R\diag(v_{i,1},\dots,v_{i,n})R^\top.
\]
Restrict attention to
\[
U=Q\diag(x_1,\dots,x_n)Q^\top,
\qquad
V=R\diag(y_1,\dots,y_n)R^\top,
\]
with $x_p,y_q>0$ and $\prod_p x_p=1$.

The commuting chart is part of the data of the restricted problem. When repeated eigenvalues are present, simultaneous diagonalizing bases need not be unique, and different choices may describe different diagonal coordinate slices. The result below optimizes only over candidates diagonal in the fixed bases $Q$ and $R$.

\begin{lemma}
	\label{lem:commutingform}
	In the above notation, the objective restricted to the fixed commuting-coordinate slice becomes
	\begin{equation}
		\label{eq:commutingobjectiveW2}
		\mcJ(x,y)
		=
		\Bigl(\sum_{p=1}^n x_p\Bigr)\Bigl(\sum_{q=1}^n y_q\Bigr)
		+
		\sum_{i=1}^N w_i\,\tr(U_i)\tr(V_i)
		-2\sum_{p=1}^n\sum_{q=1}^n c_{pq}\sqrt{x_p y_q},
	\end{equation}
	where $c_{pq}:=\sum_{i=1}^N w_i\sqrt{u_{i,p}v_{i,q}}$.
\end{lemma}

\begin{proof}
	\Cref{lem:pairwiseW2} gives, for each $i$,
	\[
	d_{\mathrm B}^2(V\ot U,V_i\ot U_i)
	=\tr(V)\tr(U)+\tr(V_i)\tr(U_i)-2\tr\bigl((V^{1/2}V_iV^{1/2})^{1/2}\bigr)\tr\bigl((U^{1/2}U_iU^{1/2})^{1/2}\bigr).
	\]
	In the fixed commuting coordinate slice, $\tr(U)=\sum_{p=1}^n x_p$ and $\tr(V)=\sum_{q=1}^n y_q$, while
	\[
	\tr\bigl((U^{1/2}U_iU^{1/2})^{1/2}\bigr)=\sum_{p=1}^n \sqrt{x_pu_{i,p}},
	\qquad
	\tr\bigl((V^{1/2}V_iV^{1/2})^{1/2}\bigr)=\sum_{q=1}^n \sqrt{y_qv_{i,q}}.
	\]
	Summing with weights $w_i$ yields
	\[
	\mcJ(x,y)
	=\Bigl(\sum_{p=1}^n x_p\Bigr)\Bigl(\sum_{q=1}^n y_q\Bigr)
	+\sum_{i=1}^N w_i\,\tr(U_i)\tr(V_i)
	-2\sum_{i=1}^N w_i\sum_{p=1}^n\sum_{q=1}^n \sqrt{x_py_q}\sqrt{u_{i,p}v_{i,q}}.
	\]
	The definition of $c_{pq}$ gives \eqref{eq:commutingobjectiveW2}.
\end{proof}

Even in one fixed commuting coordinate slice, the determinant normalization $\prod_p x_p=1$ is a nonlinear constraint in factor coordinates. The slice minimizer therefore differs from the unconstrained tensor-product mean identity: the Perron singular vectors arise from a normalized Rayleigh quotient.

\begin{remark}
	\label{rem:blockwisecaution}
	For fixed $y\in\R_{++}^n$, the expression in \eqref{eq:commutingobjectiveW2} is strictly convex as a function of $x\in\R_{++}^n$, and the same blockwise statement holds after fixing $x$. The determinant constraint $\prod_p x_p=1$ is not a Euclidean convex constraint, so uniqueness in \Cref{thm:commutingexact} is obtained instead from the Rayleigh quotient for the entrywise positive matrix $CC^\top$ and the Perron--Frobenius theorem.
\end{remark}

\begin{theorem}
	\label{thm:commutingexact}
	Assume the notation of \Cref{lem:commutingform}. Let $C=(c_{pq})\in\R_{++}^{n\times n}$ and $\sigma_1:=\sigma_{\max}(C)$.
	Choose the positive singular vector pair $u_1,v_1\in\R_{++}^n$ such that
	\[
	Cv_1=\sigma_1 u_1,
	\qquad
	C^\top u_1=\sigma_1 v_1,
	\qquad
	\norm{u_1}_2=\norm{v_1}_2=1.
	\]
	Set $\alpha:=\bigl(\prod_{p=1}^n u_{1,p}\bigr)^{-1/n}$.
	Then the fixed-coordinate-slice optimization problem
	\begin{equation}
		\label{eq:commutingchartproblem}
		\min\Bigl\{\mcJ(x,y): x\in\R_{++}^n,\ y\in\R_{++}^n,\ \prod_{p=1}^n x_p=1\Bigr\}
	\end{equation}
	has the unique minimizer within this fixed coordinate slice
	\begin{equation}
		\label{eq:exactxy}
		x_p^\star=\alpha^2 u_{1,p}^2,
		\qquad
		y_q^\star=\frac{\sigma_1^2}{\alpha^2}v_{1,q}^2.
	\end{equation}
	Equivalently, the unique minimizer in the fixed coordinate slice is
	\[
	U_\star=Q\diag(x_1^\star,\dots,x_n^\star)Q^\top,
	\qquad
	V_\star=R\diag(y_1^\star,\dots,y_n^\star)R^\top.
	\]
	The fixed-coordinate-slice minimum value is
	\begin{equation}
		\label{eq:commutingminimumvalue}
		\min_{\textup{fixed slice}} \mcJ
		=
		\sum_{i=1}^N w_i\,\tr(U_i)\tr(V_i)-\sigma_1^2.
	\end{equation}
\end{theorem}

\begin{proof}
	Introduce square-root coordinates $z_p:=\sqrt{x_p}$ and $r_q:=\sqrt{y_q}$, so that $z,r\in\R_{++}^n$ and the determinant normalization becomes $\prod_{p=1}^n z_p=1$.
	In these variables, \eqref{eq:commutingobjectiveW2} takes the form
	\begin{equation}
		\label{eq:Jzw}
		\mcJ(z,r)=\kappa+\norm{z}_2^2\norm{r}_2^2-2z^\top Cr,
	\end{equation}
	where $\kappa:=\sum_{i=1}^N w_i\,\tr(U_i)\tr(V_i)$ is constant.
	
	For fixed $z\in\R_{++}^n$, the $r$-part is a strictly convex quadratic, with $\nabla_r \mcJ(z,r)=2\norm{z}_2^2 r-2C^\top z$.
	The unique minimizer over $\R_{++}^n$ is
	\begin{equation}
		\label{eq:wstarz}
		r^\star(z)=\frac{C^\top z}{\norm{z}_2^2},
	\end{equation}
	which is positive since $C$ and $z$ are entrywise positive. Substitution produces
	\begin{equation}
		\label{eq:reducedRayleigh}
		\min_{r\in\R_{++}^n}\mcJ(z,r)
		=
		\kappa-\frac{z^\top CC^\top z}{z^\top z}.
	\end{equation}
	It remains to maximize the Rayleigh quotient of $S:=CC^\top$ over $\{z\in\R_{++}^n:\prod_p z_p=1\}$.
	
	The matrix $S$ is entrywise positive. Rayleigh--Ritz gives $z^\top Sz/(z^\top z)\le \sigma_1^2$, where $\sigma_1^2=\lambda_{\max}(S)$. Perron--Frobenius makes this eigenvalue simple and gives the unique positive unit eigenvector $u_1$. Equality holds only on the Perron ray, and the determinant constraint intersects that ray at $z^\star=\alpha u_1$, with $\alpha=(\prod_p u_{1,p})^{-1/n}$. This is the unique feasible maximizer.
	
	Using \eqref{eq:wstarz}, $r^\star=(C^\top z^\star)/\norm{z^\star}_2^2=(\sigma_1/\alpha)v_1$. Therefore $x_p^\star=(z_p^\star)^2=\alpha^2u_{1,p}^2$ and $y_q^\star=(r_q^\star)^2=\sigma_1^2v_{1,q}^2/\alpha^2$, which proves \eqref{eq:exactxy}. Substituting the maximal Rayleigh quotient value $\sigma_1^2$ in \eqref{eq:reducedRayleigh} produces \eqref{eq:commutingminimumvalue}. Since the minimization over $r$ is exact for each feasible $z$ and the Rayleigh quotient has a unique feasible maximizer, the minimizer in the fixed coordinate slice is unique.
\end{proof}

\begin{remark}
	The formula is exact and unique for the finite-dimensional problem in which both the data and the candidate are diagonal in the chosen commuting bases. It does not assert global optimality over all of $\mcK_n$ when candidates are allowed to leave that coordinate slice.
\end{remark}

\subsection{Leafwise barycenters}
The leafwise barycenter argument uses the following global scaling identity.

\begin{remark}
	\label{rem:pullbackmetric}
	The embedding $\Phi(U,V)=V\ot U$ pulls the ambient Bures metric back by
	\[
	\mathbf g_{(U,V)}\bigl((H_U,H_V),(G_U,G_V)\bigr)
	=
	g^{\mathrm B}_{V\ot U}\bigl(V\ot H_U+H_V\ot U,\,V\ot G_U+G_V\ot U\bigr),
	\]
	where $g_K^{\mathrm B}(H,G)=\frac12\ip{\mcL_K^{-1}(H)}{G}_F$ and $\mcL_K(X)=KX+XK$. The displayed pullback is obtained by substituting $D\Phi_{(U,V)}(H_U,H_V)=V\ot H_U+H_V\ot U$ into this standard Lyapunov formula for the Bures metric. At $(I_n,sI_n)$, the isotropic expression follows from $\mcL_{sI_{n^2}}(X)=2sX$ and $\tr(H_U)=0$:
	\[
	\mathbf g_{(I_n,sI_n)}\bigl((H_U,H_V),(H_U,H_V)\bigr)
	=
	\frac n4\left(s\norm{H_U}_F^2+\frac1s\norm{H_V}_F^2\right),
	\qquad \tr(H_U)=0.
	\]
\end{remark}

\begin{lemma}
	\label{lem:homotheticleaves}
	For fixed $U_\ast\in\Sp^n$ with $\det U_\ast=1$, the map $\Psi_{U_\ast}^{\mathrm{row}}(V\ot U_\ast)=V$ identifies $\mathcal F(U_\ast)$ with a homothetic copy of $(\Sp^n,d_{\mathrm B})$. More precisely,
	\[
	d_{\mathrm B}^2(V_0\ot U_\ast,V_1\ot U_\ast)
	=
	\tr(U_\ast)\,d_{\mathrm B}^2(V_0,V_1).
	\]
	For fixed $V_\ast\in\Sp^n$, the map $\Psi_{V_\ast}^{\mathrm{col}}(\alpha V_\ast\ot U)=\alpha U$ identifies $\mathcal G(V_\ast)$ with a homothetic copy of $(\Sp^n,d_{\mathrm B})$, and
	\[
	d_{\mathrm B}^2(\alpha_0V_\ast\ot U_0,\alpha_1V_\ast\ot U_1)
	=
	\tr(V_\ast)\,d_{\mathrm B}^2(\alpha_0U_0,\alpha_1U_1).
	\]
	On the isotropic leaves one recovers $\Psi_{\mathrm{row}}(\alpha I_n\ot U)=\alpha U$ and $\Psi_{\mathrm{col}}(V\ot I_n)=V$.
\end{lemma}

\begin{proof}
	The first map is bijective. The second is bijective as well since every $M\in\Sp^n$ admits the unique factorization $M=\alpha U$, where $\alpha:=(\det M)^{1/n}>0$ and $U:=\alpha^{-1}M$ has determinant one.
	
	For $\mathcal F(U_\ast)$, use the Bures formula together with
	\[
	((V_0\ot U_\ast)^{1/2}(V_1\ot U_\ast)(V_0\ot U_\ast)^{1/2})^{1/2}
	=
	\bigl(V_0^{1/2}V_1V_0^{1/2}\bigr)^{1/2}\ot U_\ast.
	\]
	Taking traces and using $\tr(V_i\ot U_\ast)=\tr(V_i)\tr(U_\ast)$ yields
	\[
	d_{\mathrm B}^2(V_0\ot U_\ast,V_1\ot U_\ast)
	=
	\tr(U_\ast)\,d_{\mathrm B}^2(V_0,V_1).
	\]
	
	For $\mathcal G(V_\ast)$, set $M_i=\alpha_iU_i$. The analogous identity is
	\[
	((V_\ast\ot M_0)^{1/2}(V_\ast\ot M_1)(V_\ast\ot M_0)^{1/2})^{1/2}
	=
	V_\ast\ot(M_0^{1/2}M_1M_0^{1/2})^{1/2}.
	\]
	The same trace calculation gives
	\[
	d_{\mathrm B}^2(\alpha_0V_\ast\ot U_0,\alpha_1V_\ast\ot U_1)
	=
	\tr(V_\ast)\,d_{\mathrm B}^2(M_0,M_1)
	=
	\tr(V_\ast)\,d_{\mathrm B}^2(\alpha_0U_0,\alpha_1U_1).
	\]
	The isotropic formulas are the special cases $U_\ast=I_n$ and $V_\ast=I_n$.
\end{proof}

Together, \Cref{lem:geodesicallyclosedleaves,lem:homotheticleaves} show that each factor leaf is geodesically closed in the ambient Bures geometry and is a homothetic copy of $(\Sp^n,d_{\mathrm B})$.

\begin{corollary}
	\label{cor:leafwisereduction}
	Suppose that the data lie in a common $U$-factor leaf, $K_i=V_i\ot U_\ast\in\mathcal F(U_\ast)$, $i=1,\dots,N$.
	Then for every $V\ot U_\ast\in\mathcal F(U_\ast)$,
	\[
	\mcJ(V\ot U_\ast)
	=
	\tr(U_\ast)\sum_{i=1}^N w_i\,d_{\mathrm B}^2(V,V_i).
	\]
	If $\bar V\in\Sp^n$ is the Bures--Wasserstein barycenter of $\{V_i\}_{i=1}^N$, then $\bar K=(\Psi_{U_\ast}^{\mathrm{row}})^{-1}(\bar V)=\bar V\ot U_\ast$ is the unique leafwise barycenter on $\mathcal F(U_\ast)$.
	
	Likewise, suppose the data lie in a common $V$-factor leaf, $K_i=\alpha_iV_\ast\ot U_i\in\mathcal G(V_\ast)$, $i=1,\dots,N$.
	Then for every $\alpha V_\ast\ot U\in\mathcal G(V_\ast)$,
	\[
	\mcJ(\alpha V_\ast\ot U)
	=
	\tr(V_\ast)\sum_{i=1}^N w_i\,d_{\mathrm B}^2(\alpha U,\alpha_iU_i).
	\]
	If $\bar M\in\Sp^n$ is the Bures--Wasserstein barycenter of $\{\alpha_iU_i\}_{i=1}^N$, then $\bar K=(\Psi_{V_\ast}^{\mathrm{col}})^{-1}(\bar M)$ is the unique leafwise barycenter on $\mathcal G(V_\ast)$.
\end{corollary}

\begin{proof}
	For the row leaf, \Cref{lem:homotheticleaves} gives $d_{\mathrm B}^2(V\ot U_\ast,V_i\ot U_\ast)=\tr(U_\ast)\,d_{\mathrm B}^2(V,V_i)$ for $i=1,\dots,N$, and hence $\mcJ(V\ot U_\ast)=\tr(U_\ast)\sum_{i=1}^N w_i\,d_{\mathrm B}^2(V,V_i)$.
	Under the bijection $\Psi_{U_\ast}^{\mathrm{row}}$, minimizers correspond as
	\[
	\argmin_{V\ot U_\ast\in\mathcal F(U_\ast)}\mcJ(V\ot U_\ast)
	=
	\bigl(\Psi_{U_\ast}^{\mathrm{row}}\bigr)^{-1}\!\left(
	\argmin_{V\in\Sp^n}\sum_{i=1}^N w_i\,d_{\mathrm B}^2(V,V_i)
	\right).
	\]
	The Bures--Wasserstein barycenter on $\Sp^n$ is unique \cite[Section~6]{BhatiaJainLim2019}. Therefore $\bar V\ot U_\ast$ is the unique leafwise minimizer.
	
	The column-leaf identity follows in the same way: $\mcJ(\alpha V_\ast\ot U)=\tr(V_\ast)\sum_{i=1}^N w_i\,d_{\mathrm B}^2(\alpha U,\alpha_iU_i)$ for every $\alpha V_\ast\ot U\in\mathcal G(V_\ast)$. The bijection $\Psi_{V_\ast}^{\mathrm{col}}(\alpha V_\ast\ot U)=\alpha U$ transports the problem to the standard barycenter problem for $\{\alpha_iU_i\}$ on $\Sp^n$, whose unique minimizer is $\bar M$.
\end{proof}

\begin{example}
	\label{ex:isotropicleaf}
	If $K_i=\alpha_i I_{n^2}$ with $\alpha_i>0$, then the barycenter on the isotropic row leaf is $\bar K=\bar\alpha I_{n^2}$, where $\sqrt{\bar\alpha}=\sum_{i=1}^N w_i\sqrt{\alpha_i}$ and $\bar\alpha=(\sum_{i=1}^N w_i\sqrt{\alpha_i})^2$. This is the one-dimensional Bures--Wasserstein barycenter obtained from \Cref{cor:leafwisereduction} with $U_i=I_n$.
\end{example}

\section{Numerical illustrations}
\label{sec:numerics}

The numerical illustrations check the formulas in finite precision. They cover the pairwise reduction, the fixed-chart departure moduli, and the Perron formula for the fixed-coordinate-slice barycenter problem. All runs use double precision, random seed 42, and 20 independent trials. In the commuting experiments, diagonal factors are generated from log-coordinates: for $\xi\in\R^n$, put $\bar\xi=n^{-1}\sum_p\xi_p$, $D_0(\xi)=\diag(e^{\xi_p-\bar\xi})_{p=1}^n$, and $D_1(\xi)=\diag(e^{\xi_p})_{p=1}^n$. Thus $\det D_0(\xi)=1$. All normal random variables below are independent standard normals unless stated otherwise, and barycenter weights are uniform.

\subsection{Pairwise reduction}
\label{subsec:exp1}

For $n \in \{8,16,32,64,128\}$, random factors are generated from $GG^\top+10^{-2}I$, with the $U$-factor determinant-normalized. Table~\ref{tab:pairwise} compares the ambient $n^2\times n^2$ Bures evaluation with the reduced formula in \Cref{lem:pairwiseW2}. For $n=128$, the ambient computation was omitted because of its storage and time cost.

\begin{table}[H]
	\centering
	\caption{Pairwise spectral reduction. Times are in seconds.}
	\label{tab:pairwise}
	\small
	\setlength{\tabcolsep}{3pt}
	\begin{tabular*}{\textwidth}{@{\extracolsep{\fill}}rccccc@{}}
			\toprule
			$n$ & ambient & reduced & ratio & rel.~err. & storage ratio \\
			\midrule
			$8$ & $5.26\times10^{-4}$~($7.50\times10^{-5}$) & $1.17\times10^{-4}$~($3.49\times10^{-5}$) & $4.49$ & $2.41\times10^{-14}$ & $32$ \\
			$16$ & $0.0409$~($3.08\times10^{-3}$) & $2.43\times10^{-4}$~($3.20\times10^{-5}$) & $169$ & $1.90\times10^{-14}$ & $128$ \\
			$32$ & $1.01$~($0.0113$) & $5.36\times10^{-4}$~($4.63\times10^{-5}$) & $1.88\times10^{3}$ & $1.15\times10^{-14}$ & $512$ \\
			$64$ & $18.1$~($0.135$) & $1.68\times10^{-3}$~($1.27\times10^{-4}$) & $1.08\times10^{4}$ & $8.97\times10^{-15}$ & $2048$ \\
			$128$ & --- & $0.0312$~($2.13\times10^{-3}$) & --- & --- & $8192$ \\
			\bottomrule
	\end{tabular*}
	\tabnote{Means over 20 trials, with standard deviations in parentheses. Relative errors are measured against the ambient computation when available; storage ratios are theoretical.}
\end{table}

\subsection{Fixed-chart geodesic closure diagnostics}
\label{subsec:exp2}

For the fixed-chart diagnostics, we take $n=32$ and use diagonal endpoints with square-root profiles $H_0=ab^\top$ and $H_1=cd^\top$. The two leaf regimes set $c=a$ or $d=b$, while the generic regime draws both factor profiles independently. Table~\ref{tab:departure} reports $\max_j\delta_{\mathrm{geo}}(j/200)$, $\max_j\delta_{\mathrm{diag}}(j/200)$, and a finite-window quadratic fit for $\delta_{\mathrm{geo}}(t)^2$ near zero. The fitted coefficient is obtained by least squares for $\delta_{\mathrm{geo}}(t_j)^2\approx c\,t_j^2$ with no intercept, using $t_j=j/200$, $j=1,\dots,10$. The factor-leaf cases vanish to machine precision, while the generic perturbations give positive moduli.

\begin{table}[htbp]
	\centering
	\caption{Fixed-chart departure moduli.}
	\label{tab:departure}
	\small
	\setlength{\tabcolsep}{4pt}
	\begin{tabular*}{\textwidth}{@{\extracolsep{\fill}}lccccc@{}}
			\toprule
			regime & $\max\delta_{\mathrm{geo}}$ & $\max\delta_{\mathrm{diag}}$ & fitted coeff. & predicted coeff. & fit rel.~err. \\
			\midrule
			shared $U$-factor leaf & $0$~($0$) & $0$~($0$) & $0$~($0$) & $0$~($0$) & $0$ \\
			shared $V$-factor leaf & $0$~($0$) & $0$~($0$) & $0$~($0$) & $0$~($0$) & $0$ \\
			generic perturbation & $1.57$~($0.24$) & $5$~($0.775$) & $32.1$~($9.54$) & $37.5$~($11.2$) & $0.146$ \\
			\bottomrule
	\end{tabular*}
	\tabnote{Means over 20 draws, with standard deviations in parentheses. In the two factor-leaf regimes, all reported moduli vanish up to machine precision.}
\end{table}

\subsection{Barycenters}
\label{subsec:exp3}

In the barycenter benchmarks, $n=8$ and $N=8$. Dataset~A has $U_i=I_n$ and $V_i=\alpha_iI_n$, $\alpha_i=e^{\zeta_i}$; Dataset~B uses $U_i=D_0(0.1\,\xi_i)$ and $V_i=\alpha_iD_1(0.1\,\eta_i)$; Dataset~C uses the same construction with perturbation size $1$. The Perron formula is compared with a bounded constrained solve in logarithmic variables $s_p=\log x_p$, $r_q=\log y_q$, subject to $\sum_p s_p=0$. The numerical solve uses SciPy's SLSQP implementation with the linear equality constraint, bounds $[-8,8]$ on all logarithmic variables, \texttt{ftol}$=10^{-12}$, and a maximum of $2000$ iterations. We measure coordinate error by $\max\{\|\widehat x-x^\star\|_2/\|x^\star\|_2,\|\widehat y-y^\star\|_2/\|y^\star\|_2\}$ and use the projected first-order residual in the same variables. Figure~\ref{fig:barycenter} and Table~\ref{tab:barycenter} report the optimization gaps, objective values, and coordinate errors.

\begin{figure}[htbp]
	\centering
	\includegraphics[width=0.70\textwidth]{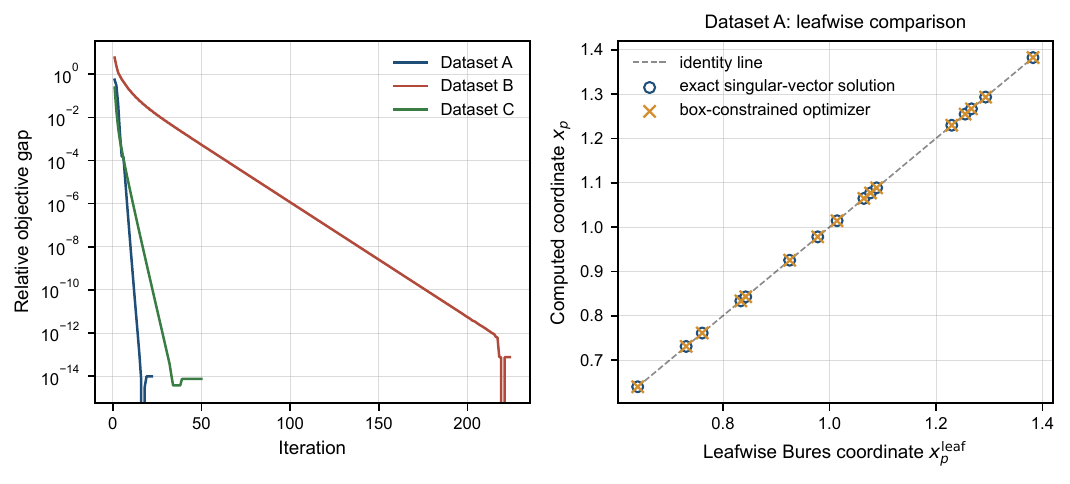}
	\caption{Fixed-coordinate-slice barycenter benchmarks: objective gaps for independent logarithmic-coordinate solves, and Dataset~A coordinates against the leafwise Bures coordinates.}
	\label{fig:barycenter}
\end{figure}

\begin{table}[htbp]
	\centering
	\caption{Fixed-coordinate-slice barycenter benchmarks.}
	\label{tab:barycenter}
	\small
	\setlength{\tabcolsep}{3pt}
	\begin{tabular*}{\textwidth}{@{\extracolsep{\fill}}clcccc@{}}
			\toprule
			dataset & description & formula obj. & numerical obj. & residual & coord. error \\
			\midrule
			A & isotropic factor leaf & $17.8$~($2.67$) & $17.8$~($2.67$) & $2.24\times10^{-7}$ & $8.71\times10^{-9}$ \\
			B & near-isotropic commuting & $1.57$~($0.51$) & $1.57$~($0.51$) & $1.15\times10^{-7}$ & $4.99\times10^{-9}$ \\
			C & generic commuting & $23.6$~($3.74$) & $23.6$~($3.74$) & $2.29\times10^{-7}$ & $8.45\times10^{-9}$ \\
			\bottomrule
	\end{tabular*}
	\tabnote{Means over 20 draws, with standard deviations in parentheses. The last two columns report the projected first-order residual and the relative coordinate error.}
\end{table}

\paragraph{Reproducibility.}
The experiments use no external data. The tables and figure are generated from the synthetic protocols above using Python~3.11.8, NumPy~1.26.4, SciPy~1.9.3, and Matplotlib~3.10.8 on a workstation with an Intel Core i7-12700H processor and 16GB RAM. 

\section{Concluding remarks}
\label{sec:con}
The determinant-normalized parametrization removes the scalar gauge from Kronecker factorizations and realizes the Kronecker positive-definite model as an embedded submanifold of the full cone. The results above show that this submanifold has very limited compatibility with ambient Bures geodesics: aside from the one-factor subfamilies, a Bures segment joining two Kronecker points leaves the model immediately. In commuting coordinates this failure is visible as the loss of rank one in the square-root profile, which also gives the departure moduli.

The partial-trace residual gives a coordinate-free form of the same obstruction. Although it is only a tangency condition for a general curve, for Bures geodesics determined by Kronecker endpoints it becomes rigid and forces the endpoint pair onto a common factor leaf. This links the fixed-chart rank criterion with the noncommuting endpoint calculation.

The barycenter formulas obtained here reflect the same restriction. Exact minimizers are available in a fixed commuting-coordinate slice through Perron singular vectors, and on common leaves through the usual Bures--Wasserstein barycenter on $\Sp^n$. A fuller understanding of barycenters on \(\mathcal K_n\) would require structural descriptions of global minimizers, exact formulas beyond these special classes, and convergence guarantees for optimization over the full determinant-normalized model.

\section*{Acknowledgements}
This research is supported by National Key R\&D Program of China (2024YFA1012401), the Science and Technology Commission of Shanghai Municipality (23JC1400501), and Natural Science Foundation of China (12241103).

\section*{Data availability}
No external data were used in this work. The numerical illustrations are reproducible from the synthetic protocols described in \Cref{sec:numerics}.



\bibliographystyle{cas-model2-names}

\bibliography{cas-refs}



\end{document}